\documentclass[12pt]{amsart}
\usepackage{amsfonts}

\usepackage[letterpaper]{geometry}

\usepackage{amssymb}
\def\R{\mathbb{R}}

\def\Z{\mathbb{Z}}
\def\N{\mathbb{N}}

\newtheorem{thm}{\bf Theorem}[section]
\newtheorem{lemma}[thm]{\bf Lemma}

\theoremstyle{definition}

\begin{document}

\title{Note on the sumset of squares}

 \author[Norbert Hegyv\'ari]{Norbert Hegyv\'ari}
 \address{Norbert Hegyv\'{a}ri, ELTE TTK,
E\"otv\"os University, Institute of Mathematics, H-1117
P\'{a}zm\'{a}ny st. 1/c, Budapest, Hungary and associated member of Alfr\'ed R\'enyi Institute of Mathematics, Hungarian Academy of Science, H-1364 Budapest, P.O.Box 127.}
 \email{hegyvari@renyi.hu}

\begin{abstract} 
It is proved that for any non-empty finite subset $Q$ of the square numbers,  
$
|Q+Q|\geq C'|Q|(\log |Q|)^{1/3+o(1)}
$.

This result essentially is proved -- with the same tools -- by Mei-Chu Chang. See in J. Funct. Anal. 207 (2004), no 2, 444-460.

So the author will withdraw this ArXiv file. 

MSC 2020 Primary 11B75,  Secondary  11E25 

Keywords: Sum of squares, Ruzsa's conjecture.
\end{abstract}

 \maketitle
\section{Introduction}

Let $S=\{1^2,2^2\dots,\}$ be the sequence of square numbers. A classical result of Landau states that $|(S+S)\cap \{1,2,\dots, n\}|\gg \frac{n}{\sqrt{\log n}}$. 

One can ask what we can say on $|Q+Q|$ where $Q$ is an arbitrary finite subset of $S$. Clearly $|Q+Q|\leq |Q|^2$ and the bound is strict as we get $Q=\{4^k\}_{1\leq k\leq n}$.

A nice and challenging conjecture of I.Z. Ruzsa states that there is a positive real number $c>0$, such that for every 
finite subset $Q$ of $S$, $|Q+Q|\geq |Q|^{1+c}$ (see e.g. [3]). We are very far to prove this conjecture and the author did  not find any quantitative lower bound for the cardinality of this sumset (see further remark in the last section).  

Nevertheless we prove:
\begin{thm}
    For any non-empty finite set $Q$ of squares we have
$$
|Q+Q|\geq C'|Q|(\log |Q|)^{1/3+o(1)}
$$
    
\end{thm}
(Here $o(1)$ is negative, i.e. the power of $\log |Q|$ is less than $1/3$.)
\section{Proof}

The proof is simple and is supported by two deep results. So let $Q\subset S$ be any finite subset of squares, and write $|Q+Q|/|Q|=K(|Q|):=K$.

The first ingredient is the Chang version of Freiman theorem (see [2]):
\begin{lemma}\label{2.1}
Assume that $A\subset \Z$ is a finite set satisfying $|A+A|\leq K|A|$. Then $A\subset P$, where $P$ is a proper
$d-$dimensional arithmetic progression with 
$$
d\leq \lfloor K-1\rfloor; \quad |P|\leq |A|\exp\{CK^2\log^3 K)\}
$$
\end{lemma}
The proper $d-$dimensional arithmetic progression is a set in the form
$$
P(x_0,x_1,\dots x_d):=\{x_0+t_1x_1+t_2x_2+\dots +t_dx_d: 0\leq t_i\leq l_i; \ i=1,2,\dots ,d\}  
$$
and proper means that $|P|=\prod_{1\leq i\leq d} l_i$.

By this lemma we get that $Q$ can be covered by a proper $d-$dimensional arithmetic progression $P$ with $|P|\leq |Q|\exp\{CK^2\log^3 K)\}$.

\begin{lemma}\label{2.2}
Let $P$ be a proper $K-$dimensional arithmetic progression. The number of squares $|Q|$ is
$$
|Q|\leq |P|^{1-\frac{2}{5K}}.
$$
\end{lemma}
\begin{proof}[Proof of lemma \ref{2.2}]
Write $P=P(x_0,x_1,\dots x_K):=\{x_0+t_1x_1+t_2x_2+\dots +t_Kx_K: 0\leq t_i\leq l_i; \ i=1,2,\dots ,K\}$. We can write $P$ as $P=x_0+P(0,x_1)+P(0,x_2)+\cdots +P(0,x_K)$. 

Write $P'=x_0+P(0,x_1)+P(0,x_2)+\cdots +P(0,x_{K-1})$. Then $P=\cup_{u\in P'}(u+P(0,x_K))$. Hence $|P|\leq |P'||P(0,x_K)|$.

Now, one of Bombieri-Zannier's deep results is that for any $1\leq i \leq K$, the number of squares in $P(0,x_i)$ is at most $x_i^{3/5+o(1)}$ (see in [1]). Thus
\begin{equation}
    |Q|\leq \prod_{1\leq j\leq K; \ i\neq j}x_jx_i^{3/5+o(1)}
\end{equation}
Multiply equation $(1)$ for every $1\leq i\leq K$, we obtain
$$
|Q|^K\leq \prod_{1\leq i\leq K}\prod_{1\leq j\leq K; \ i\neq j}x_jx_i^{3/5+o(1)}=|P|^{K-2/5+o(1)}
$$
from which we get $|Q|\leq |P|^{1-\frac{2}{5K}+o(1/K)}$ as we wanted.

Finally by Lemma \ref{2.1} we get
$$
|Q|\leq |P|^{1-\frac{2}{5K}+o(1/K)}\leq (|Q|\exp\{CK^2\log^3 K)\})^{1-\frac{2}{5K}+o(1)}
$$
and thus 
$$
|Q|^{\frac{2}{5K}+o(1/K)}\leq (\exp\{CK^2\log^3 K)\})
$$
which implies
$$
C'(\log |Q|)^{1/3+o(1)}\leq K
$$
Hence
$$
|Q+Q|\geq K|Q|\geq C'|Q|(\log |Q|)^{1/3+o(1)}
$$
which is the stated bound.
\end{proof}
\section{Concluding remarks}
1. It has not escaped my attention that in $(1)$ we gave a broad estimate of $P'$. In fact, if $K=2$, i.e. if $P=x_0+P(0,x_1)+P(0,x_2)$, then we gave a bound on $P(0,x_2)$ based on the Bombier-Zannier result. For $P(0,x_1)$ -- knowing that the set of square numbers does not contain a $4-$term arithmetic progression -- we can give a bound $r_4(x_2)$, (the maximum size of $4-$term free set up to $x_2$). Notice that $r_4(n)=O(\frac{n}{(\log n)^c})$, and this bound does not give better estimation.

\smallskip

2. Ruzsa (and maybe many others) conclude that without the Bombieri-Zannier's result the bound $|Q+Q|\geq |Q|^{1+c}$ implies that the number of squares in an arithmetic progression with length $N$ is at most $\ll N^{\frac{1}{1+c}}$ (indeed if $Q\subset \{x_0+dt:\ 1\leq t\leq N\}$ then $Q+Q\subset \{2x_0+dt:\ 2\leq t\leq 2N\}$). Hence $c>2/3$ would beat the known best result.

\smallskip

3. In 1960 Rudin conjectured, that in any $N-$term arithmetic progression the number of square numbers is $\ll \sqrt{N}$ (see in [4]). The proof of this conjecture still seems  hopeless today.  

\smallskip

4. Several conditional bounds are known for $|Q+Q|$. By the Szemer\'edi-Trotter incidence theorem one can obtain a bound 
$$
\max\{|A+A|,|A^2+A^2|\}\gg |A|^{5/4},
$$
where $A\subseteq \R$ and  $A^2=\{a^2: a\in A\}$.

\smallskip

5. A bound similar to the one in the theorem holds for any set $X\subseteq \N$ satisfying the condition that for all $x_0, N\in \N$, $|X\cap \{x_0+dt:\ 1\leq t\leq N\}|<N^\alpha$, $\alpha <1$, i.e. for every $\emptyset\neq X'\subseteq X$
$$
|X'+X'|\geq K|X'|\gg|X'|(\log |X'|)^{1/3+o(1)}.
$$
holds.

\noindent{\bf Acknowledgment.} This work was supported by the National Research, Development and Innovation Fund of Hungary through project no. grant K-146387.

\end{document}